\documentclass{amsart} 
\usepackage{amsmath,amsthm}
\usepackage{amssymb,amscd,latexsym}
\usepackage{boxedminipage}
\usepackage{eepic}
\usepackage{epic}
\usepackage{wrapfig}

\newcommand{\erase}[1]{}
\theoremstyle{remark}
\newtheorem{remark}{Remark} 
\newtheorem{theorem}{Theorem}[section]
\newtheorem{definition}[theorem]{Definition}
\newtheorem{proposition}[theorem]{Proposition}

\newtheorem{corollary}[theorem]{Corollary}

\newtheorem{lemma}[theorem]{Lemma}
\numberwithin{equation}{section}
\newcommand{\bp}{\begin{pmatrix}}
\newcommand{\ep}{\end{pmatrix}}
\newcommand{\bps}{\begin{smallmatrix}}
\newcommand{\eps}{\end{smallmatrix}}
\def\C{{\mathbb C}}

\def\R{{\mathbb R}}
\def\Z{{\mathbb Z}}

\def \0{{\bf 0}}
\def \1{{\bf 1}}

\def \mf#1#2#3#4{
\xymatrix{{#1}\  \ar@<0.4ex>[r]^{{#2}} & \ {#4}
\ar@<0.4ex>[l]^{{#3}}
}
}

\def \mfs#1#2#3#4{\!
\xymatrix@C=1.5em{{#1} \! \ar@<0.2ex>[r]^{{#2}} & \! {#4}
\ar@<0.2ex>[l]^{{#3}}
}
\!}

\def \mfl#1#2#3#4{
\xymatrix@C=2.6em{{#1}\  \ar@<0.4ex>[r]^{{#2}} &\  {#4}
\ar@<0.2ex>[l]^{{#3}}
}
}

\def \mfss#1#2#3#4{\!
\xymatrix@C=1.5em{{#1} \ar@<0.3ex>[r]^{{#2}} & {#4}
\ar@<0.3ex>[l]^{{#3}}
}
\!}

\newcommand{\deeq}{\mathbin{\hbox{$=$ \lower 1.7pt\rlap{\hskip -8.5pt .}}}} 

\begin{document}
\title{The skew growth functions $N_{M, \mathrm{deg}}(t)$ for the monoid of type $\mathrm{B_{ii}}$ and others.} 

\author{Tadashi Ishibe}
\begin{abstract}
Let $M$ be a positive homogeneously finitely presented monoid ${\langle L \mid R\,\rangle}_{mo}$ that satisfies the cancellation condition and is equipped with the degree map $\deg\!:\!M \to\!\Z_{\ge0}$ defined by assigning to each equivalence class of words the length of the words, and let $P_{M, \deg}(t)\!:=\!\!\sum_{u\!\in\! M}\!t^{\deg(u)}$ be its generating series, called the (\emph{spherical}) \emph{growth function}. If $M$ satisfies the condition that any subset $J$ of $I_0$ ($:=$ the image of the set $L$ in $M$) admits either the least right common multiple $\Delta_{J}$ or no common multiple in $M$, then the inversion function $P_{M, \deg}(t)^{-1}$ is given by the polynomial $\sum_{J \subset I_{0}}(-1)^{\#J} t^{\deg(\Delta_{J})}$, where the summation index $J$ runs over all subsets of $I_0$ whose least right common multiple exists. Since a monoid $M$, in general, may not admit the least right common multiple $\Delta_{J}$ for a given subset $J$ of it, if we attempt to generalize the formula, the consideration to obtain the above formula is invalid. In order to resolve this obstruction, we will examine the set $\mathrm{mcm}(J)$ of minimal common right multiples of $J$. Then, we need to introduce a concept of a \emph{tower of minimal common multiples} of elements of $M$ and denote the set of all the towers in $M$ by $\mathrm{Tmcm}(M)$. By considering the structure of the set $\mathrm{Tmcm}(M)$ of all the towers, K. Saito has proved the \emph{inversion formula}
\[
P_{M,\deg}(t)\cdot N_{M,\deg}(t)=1,
\]
where the second factor in LHS is a suitably signed generating series
\[
N_{M,\deg}(t):= 1 + \sum_{T\in \mathrm{Tmcm}(M)}(-1)^{\#J_1+\cdots+\#J_{n}-n+1}\sum_{\Delta\in \mathrm{mcm}(J_n)} t^{\deg(\Delta)},
\]
called the \emph{skew growth function}.\par
In this article, we will present several explicit calculations of examples of the skew growth functions for the monoid of type $\mathrm{B_{ii}}$ and others whose towers do not stop on the first stage $J_1$. For monoids of this kind, on one hand, we generally find difficulty in calculating the skew growth functions, since, for any tower $T=(I_0, J_1, J_2, \cdots, J_n)$, we need to calculate the set $\mathrm{mcm}(J_{i})$ explicitly for each $J_{i}$. On the other hand, we find difficulty in showing the cancellativity of them, since the pre-existing technique is far from perfect. By improving the technique, we will show the cancellativity of the above examples successfully.
\end{abstract}

  







\maketitle

\section{Introduction} 
Let $M$ be a positive homogeneously finitely presented monoid ${\langle L \mid R\,\rangle}_{mo}$ that satisfies the cancellation condition (i.e. $axb = ayb$ implies $x = y$ ). Due to the homogeneity, we naturally define a map $\deg\!:\!M \to\!\Z_{\ge0}$ defined by assigning to each equivalence class of words the length of the words. By analogy to the (spherical) growth function for a group, the (\emph{spherical}) \emph{growth function} for a monoid $M$ is defined as $P_{M, \deg}(t)\!:=\!\!\sum_{u\!\in\! M}\!t^{\deg(u)}$ and has been discussed and calculated by several authors (\cite{[A-N]}\cite{[B]}\cite{[Bro]}\cite{[De]}\cite{[I1]}\cite{[S2]}\cite{[S3]}\cite{[S4]}\cite{[S5]}\cite{[Xu]}). In these studies except the \cite{[I1]}, the monoid $M$ satisfies the {\bf condition $\mathcal{L}$} that any subset $J$ of $I_0$ ($:=$ the image of the set $L$ in $M$) admits either the least right common multiple $\Delta_{J}$ or no common multiple in $M$. Then, the inversion function $P_{M, \deg}(t)^{-1}$ is given by the following function, called the \emph{skew growth function}, 
\[
N_{M, \deg}(t):= \sum_{J \subset I_{0}}(-1)^{\#J} t^{\deg(\Delta_{J})},
\]
where the summation index $J$ runs over all subsets of $I_0$ whose least right common multiple exists. Thus, the growth function $P_{M, \deg}(t)$ can be calculated through the skew growth function $N_{M, \deg}(t)$. We remark that, in this case, $N_{M, \deg}(t)$ is a polynomial.\par
In general, a monoid $M$ may not admit the least right common multiple $\Delta_{J}$ for a given subset $J$ of it. Indeed, the monoid, called the type $\mathrm{B_{ii}}$, that was discussed in \cite{[I1]}, does not satisfy the {\bf condition $\mathcal{L}$}. Nevertheless, the author has made a success in calculating the spherical growth function for it. The spherical growth function can be expressed as a rational function $\frac{1-t+t^2}{(1-t)^4}$. Since the form of the numerator polynomial $1-t+t^2$ is suggestive, a generalization of the inversion formula to a wider class of monoids which do not always satisfy the {\bf condition $\mathcal{L}$} is naturally expected. Instead of considering the least common multiple $\Delta_{J}$, we study the set $\mathrm{mcm}(J)$ of \emph{minimal common right multiples} of $J$. However, the datum $\{ \mathrm{mcm}(J) \}_{J\subset I_{0}}$ is still not sufficient to recover the inversion formula, since a subset $J'$ of $\mathrm{mcm}(J)$ in general may have common right multiples. Thus, we need to consider the set $\mathrm{mcm}(J')$ for a subset $J'$ of $\mathrm{mcm}(J)$. Then, we may again need to consider $\mathrm{mcm}(J'')$ for a subset $J''$ of $\mathrm{mcm}(J')$, and so on. Repeating this process, we are naturally led to consider a notion of \emph{tower}: a finite sequence $I_{0} \supset J_{1},$$J_{2}, \ldots, J_{n}$ of subsets of $M$ such that $J_{2}\subset \mathrm{mcm}(J_{1})$, \ldots, $J_{n}\subset \mathrm{mcm}(J_{n-1})$.\par
 In \cite{[S1]}, K. Saito has made a success in generalizing the inversion formula for a rather wider class of monoids. Namely, for a cancellative monoid $M$ equipped with a discrete degree map $\deg\!:\!M \to\!\R_{\ge0}$ (see \cite{[S1]}\S4), he defined the (spherical) growth function $P_{M,\deg}(t)$ for $M$ with respect to $\deg$, and, by considering the set $\mathrm{Tmcm}(M)$ of all towers $T=(I_0, J_1, J_2, \cdots, J_n)$ in $M$ (we do not explain here what the set $I_0$ is about, cf. \cite{[S1]}), he defined the skew growth function 
\[
N_{M,\deg}(t):= 1 + \sum_{T\in \mathrm{Tmcm}(M)}(-1)^{\#J_1+\cdots+\#J_{n}-n+1}\sum_{\Delta\in \mathrm{mcm}(J_{n})} t^{\deg(\Delta)}.
\]
Then, he has shown the inversion formula for $M$ with respect to $\deg$
\[
P_{M,\deg}(t)\cdot N_{M,\deg}(t)=1.
\]
Thus, when we put $h(M, \mathrm{deg}):=\max\{n \mid T=(I_0, J_1, J_2, \cdots, J_n) \in \mathrm{Tmcm}(M)\}$, the inversion formula covers all the cases $\infty \ge h(M, \mathrm{deg})\ge0$.
\par
In this article, for a positive homogeneously finitely presented cancellative monoid $M = {\langle L \mid R\,\rangle}_{mo}$ that do not satisfy the {\bf condition $\mathcal{L}$}, we will present several explicit calculations of examples of the skew growth functions. For a non-abelian monoid $M = {\langle L \mid R\,\rangle}_{mo}$ whose $h(M, \mathrm{deg})$ is equal to $\infty$, one may think that calculations of the skew growth functions are not practicable. However, in \S5, we succeed in the non-trivial calculation for the monoid of type $\mathrm{B_{ii}}$, partially because, for any tower $T=(I_0, J_1, J_2, \cdots, J_n)$, we can calculate the set $\mathrm{mcm}(J_{i})$ explicitly for each $J_{i}$ due to the Lemma 5.5. For the same reason, for the monoids $G^{+}_n$ and $H^{+}_n$ whose $h(M, \mathrm{deg})$ is equal to $2$ and the abelian monoid $M_{\mathrm{abel}}$ whose $h(M, \mathrm{deg})$ is equal to $\infty$, we can calculate the skew growth functions in \S5. As far as we know, for non-abelian monoids that do not satisfy the {\bf condition $\mathcal{L}$}, there are few examples for which the cancellativity of them has been shown, since the pre-existing technique to show the cancellativity is far from perfect (\cite{[G]}\cite{[B-S]}\cite{[Deh1]}\cite{[Deh2]}). For calculations of the skew growth functions, improvement of the technique to show the cancellativity is expected. In \cite{[Deh1]}, \cite{[Deh2]}, if presentation of a positive homogeneously presented monoid satisfies some condition, called \emph{completeness}, the cancellativity of it can be trivially checked. However, in general, the presentation of a monoid is not complete. When the presentation is not complete, in order to obtain a complete presentaion, some procedure, called \emph{completion}, is carried out. From our experience, for most of non-abelian monoids that do not satisfy the {\bf condition $\mathcal{L}$}, these procedures do not finish in finite steps. For monoids of this kind, nothing is discussed in \cite{[Deh1]}, \cite{[Deh2]}. Thus, we attempt improving the technique for this class of monoids. On the other hand, the presentations of the examples $G^{+}_{\mathrm{B_{ii}}}$ (\cite{[I1]}), $G^{+}_{m, n}$ (\cite{[I2]}), $G^{+}_n$ and $H^{+}_n$ (\S3) are not complete and the procedures do not finish in finite steps. Nevertheless, we show the cancellativity of them successfully (\cite{[I1]}, \cite{[I2]}), by improving the technique. In \S4, by the improved technique, we also show the cancellativity. In \cite{[S2]}\cite{[S3]}\cite{[S4]}\cite{[S5]}\cite{[K-T-Y]}, the distribution of the zeroes of the denominator polynomials of the growth functions are investigated from the viewpoint of limit functions (\cite{[S4]}). In \S6, we will explore another viewpoint from an analysis of the three examples. Namely, we will present three examples that suggest the ralationship between the form of the spherical growth function for a positive homogeneously finitely presented cancellative monoid ${\langle L \mid R\,\rangle}_{mo}$ and properties of the corresponding group ${\langle L \mid R\,\rangle}$.
\section{Positive homogeneous presentation}

In this section, we first recall from \cite{[S-I]}, \cite{[B-S]} some basic definitions and notations. Secondly, for a positive homogeneously finitely presented group 
\[
G = \langle L \!\mid \!R\,\rangle, 
\]
we associate a monoid defined by it. We will prepare basic definitions in a positive homogeneously presented monoid. Lastly, we define two operations on the set of subsets of a monoid.\par
First, we recall from \cite{[S-I]} basic definitions on a monoid $M$.\par
\begin{definition}
{\it 
1. A monoid $M$ is called \emph{cancellative}, if a relation $AXB = \! AYB$
 for $A,B,X,Y\!\in M$ implies $X = Y$.\par

2. For two elements $u, v$ in $M$, we denote
\[
u|_{l}v
\]
if there exists an element $x$ in $M$ such that $v = ux$. We say that $u$ divides $v$ from the left, or, $v$ is a right-multiple of $u$.\par

3. We say that $G^+$ is \emph{conical}, if $1$ is the only invertible element in $G^+$.}
\end{definition}
 Next, we recall from \cite{[B-S]} some terminologies and concepts. Let $L$ be a finite set. Let $F(L)$ be the free group generated by $L$, and let $L^*$ be the free monoid generated by $L$ inside $F(L)$. We call the elements of $F(L)$ \emph{words} and the elements of $L^*$ \emph{positive words}. The empty word $\varepsilon$ is the identity element of $L^*$.
If two words $A$, $B$ are identical letter by letter, we write $A \equiv B$.
 Let $G = \langle L \!\mid \!R\,\rangle$ be a positive homogeneously presented group (i.e. the set $R$ of relations consists of those of the form $R_i\!=\!S_i$ 
where 
$R_i$ and $S_i$ are positive words of the same length
 ), where $R$ is the set of relations. We often denote the images of the letters and words under the quotient homomorphism 
$\ F(L)\ \longrightarrow\ G$
by the same symbols and the equivalence relation on elements $A$ and $B$ in $G$ is denoted by $A = B$.

Next, we recall from \cite{[S-I]}, \cite{[I1]} some basic concepts 
on positive homogeneously presented monoid.
\ \ \\
\begin{definition}
{\it Let 
$G = \langle L \!\mid \!R\,\rangle$
be a positive homogeneously finitely presented group, where $L$ is the set of generators 
(called alphabet) and 
$R$ is the set of relations. 
Then we associate a monoid $G^+ = {\langle L \mid R\,\rangle}_{mo}$ defined as the quotient 
of the free monoid $L^*$ generated by $L$ by the equivalence relation defined as follows:  \par 
$\mathrm{i})$ two words $U$ and $V$ in $L^*$ are called \emph{elementarily 
equivalent} if either $U \equiv V$ or $V$ is obtained from $U$ by substituting 
a substring $R_i$ of $U$ by $S_i$ where $R_i\!=\!S_i$ is a relation of $R$ 
($S_i=R_i$ is also a relation if $R_i=S_i$ is a relation), \par 
$\mathrm{ii})$ two words $U$ and $V$ in $L^*$ are called \emph{equivalent}, denoted by $U \deeq V$, if there exists a sequence $U\! \equiv \!W_0, W_1,\ldots, W_n\! \equiv \!V $ of words in $L^*$ for $n\!\in\!\Z_{\ge0}$
such that $W_i$ is elementarily equivalent to $W_{i-1}$ for $i=1,\ldots,n$.\par
Due to the homogeneity of the relations, we define a homomorphism:
$$\mathrm{deg}\ : \ G^+\ \longrightarrow\ \Z_{\ge0}$$
by assigning to each equivalence class of words the length of the words.}\par

\end{definition}

\begin{remark}{\it For a positive homogeneously presented group $G = \langle L \!\mid \!R\,\rangle$, the associated monoid $G^+ = {\langle L \mid R\,\rangle}_{mo}$ is conical.}
\end{remark}

\begin{remark}{\it In \cite{[S1]}, for a monoid $M$, the quotient set $M/\!\sim$ is considered, where the equivalence relation $\sim$ on $M$ is defined by putting $u \sim v$\,\,\,$\Leftrightarrow_{def.}$\,\, $u |_{l} v\, \&\, v |_{l} u$. Due to the conicity, if $M = {\langle L \mid R\,\rangle}_{mo}$, then we say that $M/\!\sim \,\,= M$.}
\end{remark}

Lastly, we consider two operations on the set of subsets of a monoid $M$. For a subset $J$ of $M$, we put
\[
\mathrm{cm}_{r}(J) := \{ u \in M \mid j \,|_l \,u ,\, \forall j \in J \},\,\,\,\,\,\,\,\,\,\,\,\,\,\,\,\,\,\,\,\,\,\,\,\,\,\,\,\,
\]
\[
\mathrm{min}_{r}(J) := \{ u \in J \mid \exists v \in J \,\,\mathrm{s.t.}\,\, v \,|_l \,u \Rightarrow v = u \},
\]
and their composition: the set of \emph{minimal common multiples} of the set $J$ by
\[
\mathrm{mcm}(J) := \mathrm{min}_{r}(\mathrm{cm}_{r}(J)).
\]
\section{Generating functions $P_{M, \mathrm{deg}}$ and $N_{M, \mathrm{deg}}$}

In this section, for a positive homogeneous presented cancellative monoid 
\[
M = {\langle L \mid R\,\rangle}_{mo},
\]
 we define a growth function $P_{M, \mathrm{deg}}$ and a skew growth function $N_{M, \mathrm{deg}}$. Next, we recall from \cite{[S1]} the inversion formula for the growth function of $M$.\par
First, we introduce a concept of towers of minimal common multiples in $M$.\par
\begin{definition}
{\it A \emph{tower}  of $M$ of height $n \in\Z_{\ge0}$ is a sequence 
\[
T:=(I_0, J_1, J_2, \cdots, J_n) 
\]
of subsets of $M$ satisfying the followings.\\
$\mathrm{i})$  $I_0 :=$ \text{the image of the set} $L$ \text{in} $M$.\\
$\mathrm{ii})$  $\mathrm{mcm}(J_k)\not=\emptyset$   \text{ and we put} $I_k:=\mathrm{mcm}(J_k)$ \text{for}  $k=1, \ldots, n$.\\
$\mathrm{iii})$  $J_k \subset I_{k-1}$  \text{such that}  $1 < \sharp J_k < \infty$  \text{for} $k=1, \ldots, n$. 

Here, we call $I_0$,  $J_k$ and $I_k$, the {\it ground}, the  $k$th {\it stage}  and  the {\it set of minimal common multiples on the $k$th stage} of the tower $T$, respectively. In particular, the set of minimal common multiples on the top stage is denoted by $|T|:= I_{n}$.

The set of all towers of $M$  shall be denote by $\mathrm{Tmcm}(M)$. We put $h(M, \mathrm{deg}):=\max\{\mathrm{height}\,\, \mathrm{of}\,\, T \in \mathrm{Tmcm}(M)\}$.   
}
\end{definition}
\begin{remark} $\mathrm{i})$ It is clear that $M$ is a free monoid if and only if $h(M, \mathrm{deg}) = 0$.\\
$\mathrm{ii})$ All of the monoids discussed in \cite{[A-N]}, \cite{[B-S]}, \cite{[S2]}, \cite{[S3]} have $h(M, \mathrm{deg}) \le 1$.\\
$\mathrm{iii})$ For the following cancellative monoid
\[
\begin{array}{lll}
G^+_{\mathrm{B_{ii}}}:=
\biggl{\langle}
a,b,c\,
\biggl{|}
\begin{array}{lll}cbb=bba,\\
 ab=bc,\\
 ac=ca
  \end{array}
\biggl{\rangle}_{mo} ,
\end{array}
\]

we have $h(G^+_{\mathrm{B_{ii}}}, \mathrm{deg}) = \infty$ (\S5).\\
$\mathrm{iv})$ For the following cancellative monoids
\[
\begin{array}{lll}
G^{+}_{n}:=
\biggl{\langle}
a,b,c\,
\biggl{|}
\begin{array}{cc}cb^{n}=b^{n}a,\\
 ab=bc,\\
 ac=ca
  \end{array}
\biggl{\rangle}_{mo} (n=3, 4, \ldots),
\end{array}
\]

\[
\begin{array}{lll}
H^{+}_{n}:=
\biggl{\langle}
a,b,c\,
\biggl{|}
\begin{array}{cc}b(ab)^{n}ba=cb(ab)^{n}b,\\
 ab=bc,\\
 ac=ca
  \end{array}
\biggl{\rangle}_{mo} (n=1, 2, \ldots),
\end{array}
\]

we have $h(G^{+}_{n}, \mathrm{deg}) = 2$ and $h(H^{+}_{n}, \mathrm{deg}) = 2$ (\S5).\\
$\mathrm{v})$ For the following abelian cancellative monoid
\[
\begin{array}{lll}
M_{\mathrm{abel}, m}:=
\biggl{\langle}
a,b\,
\biggl{|}
\begin{array}{cc}a^{m}=b^{m},\\
 ab=ba\\
  \end{array}
\biggl{\rangle}_{mo} (m = 2, 3, \ldots),
\end{array}
\]

we have $h(M_{\mathrm{abel}, m}, \mathrm{deg}) = \infty$ (\S5).\\
\end{remark}
Secondly, we define a growth function $P_{M, \mathrm{deg}}$ and a skew growth function $N_{M, \mathrm{deg}}$. In the previous section, we have fixed a degree map $\mathrm{deg}$ on $M$. Then, we define the \emph{growth function of the monoid $(M, \mathrm{deg})$} by
\[
P_{M, \mathrm{deg}} := \sum_{u \in M} t^{\deg(u)} = \sum_{d \in \Z_{\ge0}} \#(M_{d})\  t^{d},
\]
where we put $M_{d}:=\{u \in M \mid \deg(u)=d \}$. And we define the \emph{skew growth function of the monoid $(M, \mathrm{deg})$} by
\begin{equation}
N_{M,\deg}(t):= 1 + \sum_{T\in \mathrm{Tmcm}(M)}(-1)^{\#J_1+\cdots+\#J_{n}-n+1}\sum_{\Delta\in |T|} t^{\deg(\Delta)}. 
\end{equation}
\begin{remark}
{\it In the definition $(3.1)$, we can write down the coefficient of the term $t$ directly. Namely, we write 
\[
N_{M,\deg}(t)= 1- \#(I_{0}) t + \sum_{\mathrm{height}\,\, \mathrm{of}\,\, T\,\, \ge1}(-1)^{\#J_1+\cdots+\#J_{n}-n+1}\sum_{\Delta\in |T|} t^{\deg(\Delta)}. 
\]
} 
\end{remark}
\begin{remark}
Therefore, if $M$ is a free monoid of rank $n$, then we have $N_{M,\deg}(t)= 1- nt$.
\end{remark}
\par
Lastly, we recall from  \cite{[S1]} the inversion formula for the growth function of $(M,\deg)$.
\begin{theorem}
{\it We have the inversion formula
\[
P_{M,\deg}(t)\cdot N_{M,\deg}(t)=1. 
\]
}
\end{theorem}
\section{Cancellativity of $G^{+}_{n}$ and $H^{+}_{n}$}
In this section, for a preparation for calculations of the skew growth functions for the monoids $G^{+}_{n}$ and $H^{+}_{n}$ in \S5, we prove the cancellativity of them.\par 
First, we show the cancellativity of $G^{+}_{n}$.
\ \ \\
\begin{theorem}
{\it The monoid $G_{n}^{+}$ is a cancellative monoid. }
 \par 
\end{theorem}
\begin{proof} First, we remark the following.\par
\begin{proposition}
{\it The left cancellativity on $G_{n}^{+}$ implies the right cancellativity.} 
 \par 
\end{proposition}
\begin{proof} Consider a map
 $\varphi:G^{+}_{n}\rightarrow G^{+}_{n}$,
 $W\mapsto \varphi(W):=\sigma$$(rev(W))$, where $\sigma$ is a
 permutation $\big(^{\,a, b, c}_{\,
c, b, a}\big)$ and $rev(W)$
 is the reverse of the word $W=x_1 x_2 \cdots x_k$ ($x_i$ is a letter) given by the word  $x_k  \cdots x_2 x_1$. In
 view of the defining relation of $G^+_{n}$, $\varphi$ is well-defined and is an anti-isomorphism. If $\beta \alpha \deeq \!\! \gamma \alpha$, then $\varphi(\beta \alpha) \deeq \! \varphi(\gamma \alpha)$, i.e.,  $\varphi(\alpha) \varphi(\beta) \deeq \varphi(\alpha)\varphi(\gamma)$. Using the left cancellativity, we obtain $\varphi(\beta) \deeq \! \varphi(\gamma)$ and, hence, $\beta \deeq \! \gamma $.

\end{proof}
The following is sufficient to show the left cancellativity on $G^+_{n}$.

\begin{proposition}
{\it Let $Y$ be a positive word in $G^+_{n}$ of length $r\in \Z_{\ge0}$ and let $X^{(h)}$ be a positive word in $G^+_{n}$ of length $r-h \in \{\,r-n+1, \ldots, r \,\}$.
\smallskip
\\
{\rm (i)}\, If $vX^{(0)} \deeq \! vY$ for some $v \in \{a, b, c \}$, then $X^{(0)} \deeq \! Y$.\\
{\rm (ii)}\, If $a X^{(0)} \deeq b Y$, then $X^{(0)} \deeq b Z$, $Y \deeq c Z$ for some positive word $Z$.\\
{\rm (iii)}\, If $a X^{(0)} \deeq c Y$, then $X^{(0)} \deeq c Z$, $Y \deeq a Z$ for some positive word $Z$.\\
{\rm ({iv}\,-\,$0$)}\, If $b X^{(0)} \deeq c Y$, then there exist an integer $k$ $(0 \leq k \leq r-n)$ and a positive word $Z$ such that $X^{(0)} \deeq c^k b^{n-1}a \cdot Z$ and $Y \deeq  a^k b^n \cdot Z$.\\
{\rm (iv\,-\,$1$\,-\,$a$)}\, There does not exist words $X^{(1)}$ and $Y$ that satisfy an equation $ba \cdot X^{(1)} \deeq c Y$. \\
{\rm (iv\,-\,$1$\,-\,$b$)}\, If $bb \cdot X^{(1)} \deeq c Y$, then $X^{(1)} \deeq b^{n-2}a \cdot Z$ and $Y \deeq b^n \cdot Z$ for some positive word $Z$. \\
{\rm (iv\,-\,$1$\,-\,$c$)}\, If $bc \cdot X^{(1)} \deeq c Y$, then there exist an integer $k$ $(0 \leq k \leq r-n-1)$ and a positive word $Z$ such that $X^{(1)} \deeq c^{k}b^{n-1}a \cdot Z$ and $Y \deeq a^{k-1} b^n \cdot Z$. \\
If $n \ge 4$, then, for $2 \leq h \leq n-2$, we have to prepare the following propositions {\rm (iv\,-\,$h$\,-\,$a$)}\, {\rm (iv\,-\,$h$\,-\,$b$)}\ and {\rm (iv\,-\,$h$\,-\,$c$)}\:\\
{\rm (iv\,-\,$h$\,-\,$a$)}\, There does not exist positive words $X^{(h)}$ and $Y$ that satisfy an equation $b^{h}a \cdot X^{(h)} \deeq c Y$. \\
{\rm (iv\,-\,$h$\,-\,$b$)}\, If $b^{h+1} \cdot X^{(h)} \deeq c Y$, then $X^{(h)} \deeq b^{n-h-1}a \cdot Z$ and $Y \deeq b^n \cdot Z$ for some positive word $Z$.\\
{\rm (iv\,-\,$h$\,-\,$c$)}\, There does not exist positive words $X^{(h)}$ and $Y$ that satisfy an equation $b^{h}c \cdot X^{(h)} \deeq c Y$. \\
{\rm (iv\,-\,($n-1$)\,-\,$a$)}\, If $b^{n-1}a \cdot X^{(n-1)} \deeq c Y$, then $X^{(n-1)} \deeq ba \cdot Z$ and $Y \deeq b^n c \cdot Z$ for some positive word $Z$. \\
{\rm (iv\,-\,($n-1$)\,-\,$b$)}\, If $b^n \cdot X^{(n-1)} \deeq c Y$, then $X^{(n-1)} \deeq a Z$ and $Y \deeq b^n \cdot Z$ for some positive word $Z$. \\
{\rm (iv\,-\,($n-1$)\,-\,$c$)}\, There does not exist positive words $X^{(n-1)}$ and $Y$ that satisfy an equation $b^{n-1}c \cdot X^{(n-1)} \deeq c Y$.} \\
 \par 
\end{proposition}
\begin{proof} We will show the general theorem, by refering to the triple induction (see \cite{[I2]}). The theorem for a positive word $Y$ of word-length $r$ and $X^{(h)}$ of word-length $r-h \in \{\,r-n+1, \ldots, r \,\} $ will be refered to as $\mathrm{H}_{r, h}$. It is easy to show that, for $r = 0, 1$, $\mathrm{H}_{r, h}$ is true. If a positive word $U_1$ is transformed into $U_2$ by using $t$ single applications of the defining relations of $G^+_{n}$, then the whole transformation will be said to be of \emph{chain-length} $t$. For induction hypothesis, we assume\\
$(\mathrm{A})$\,\,$\mathrm{H}_{s, h}$ is true for $s = 0, \ldots, r$ and arbitrary $h$ for transformations of all chain-lengths, \\
and\\
$(\mathrm{B})$\,\,$\mathrm{H}_{r+1, h}$ is true for $0$ $\leq$ $h$ $\leq$ $n-1$ for all chain-lengths $\leq$ $t$.\\
We will show the theorem $\mathrm{H}_{r+1, h}$ for chain-lengths $t+1$. For the sake of simplicity, we devide the proof into two steps.\\
{\bf Step 1}.\,$\mathrm{H}_{r+1, h}$ for $h = 0$\\
Let $X, Y$ be of word-length $r+1$, and let
\[
v_1 X \deeq v_2 W_2 \deeq \cdots \,\deeq v_{t+1} W_{t+1} \deeq v_{t+2} Y
\]
be a sequence of single transformations of $t+1$ steps, where $v_1, \ldots, v_{t+2} \in \{\,a, b, c\, \}$ and $W_2, \ldots , W_{t+1}$ are positive words of length $r+1$. By the assumption $t > 1$, there exists an index $\tau \in \{\,2, \ldots, t+1\,\}$ such that we can decompose the sequence into two steps
\[
v_1 X \deeq v_{\tau} W_{\tau} \deeq v_{t+2} Y,
\]
in which each step satisfies the induction hypothesis $(\mathrm{B})$.\par
If there exists $\tau$ such that $v_{\tau}$ is equal to either to $v_1$ or $v_{t+2}$, then by induction hypothesis, $W_{\tau}$ is equivalent either to $X$ or to $Y$. Hence, we obtain the statement for the $v_1 X \deeq v_{t+2} Y$. Thus, we assume from now on $v_{\tau} \not= v_1, v_{t+2}$ for $1 < \tau \leq t+1$.\par
Suppose $v_1 = v_{t+2}$. If there exists $\tau$ such that $(\,v_1 = v_{t+2}, v_{\tau}\, ) \not= (\,b, c\, ), (\,c, b\, )$, then each of the equivalences says the existence of $\alpha, \beta \in \{\,a, b, c\, \}$ and positive words $Z_1, Z_2$ such that $X \deeq \alpha Z_1$, $W_{\tau} \deeq \beta Z_1 \deeq \beta Z_2$ and $Y \deeq \alpha Z_2$. Applying the induction hypothesis $(\mathrm{A})$ to $\beta Z_1 \deeq \beta Z_2$, we get $Z_1 \deeq Z_2$. Hence, we obtain the statement $X \deeq \alpha Z_1 \deeq \alpha Z_2 \deeq Y$. Thus, we exclude these cases from our considerations. Next, we consider the case $(\,v_1 = v_{t+2}, v_{\tau}\, ) = (b, c )$. However, because of the above consideration, we say $v_2 = \cdots = v_{t+1} = c\,$. Hence, we consider the following case
\[
b X \deeq c W_1 \deeq \cdots \deeq c W_{t+1} \deeq b Y.
\]
Applying the induction hypothesis $(\mathrm{B})$ to each step, we say that there exist positive words $Z_3$ and $Z_4$ such that
\[
X \deeq b^{n-1}a \cdot Z_3,\,\, W_1 \deeq b^n \cdot Z_3,\,\,\,\,
\]
\[
W_{t+1} \deeq b^n \cdot Z_4,\,\, Y \deeq  b^{n-1}a \cdot Z_4.
\]
Since an equation $W_1 \deeq W_{t+1}$ holds, we say that 
\[
b^n \cdot Z_3 \deeq b^n \cdot Z_4.
\]
By induction hypothesis, we have $X \deeq Y$.\\
In the case of $(\,v_1 = v_{t+2}, v_{\tau}\, ) = (c, b)$, we can prove the statement in a similar manner.\\
Suppose $v_1 \not= v_{t+2}$. We consider the following three cases.
\par
Case 1\,: $(\,v_1, v_{\tau}, v_{t+2}\, ) = (\,a, b, c\, )$\\
Because of the above consideration, we consider the case $\tau = t+1$, namely
\[
a X \deeq b W_{t+1} \deeq c Y.
\]
Applying the induction hypothesis to each step, we say that there exist positive words $Z_1$ and $Z_{2}$ such that
\[
X \deeq b Z_1,\,\, W_{t+1} \deeq c Z_1,\,\,\,\,\,\,\,\,\,\,\,\,\,\,
\]
\[
\,\,\,\,\,\,\,\,\,W_{t+1} \deeq b^{n-1}a \cdot Z_{2},\,\, Y \deeq b^n \cdot Z_{2}.
\]
Thus, we say that $c \cdot Z_1 \deeq  b^{n-1}a \cdot Z_{2}$. Applying the induction hypothesis $(\mathrm{A})$ to this equation, we say that there exists a positive word $Z_3$ such that
\[
 Z_1 \deeq b^{n}c \cdot Z_3,\,\, Z_2 \deeq ba \cdot Z_3.
\]
Hence, we have $X \deeq c b^{n+1} \cdot Z_3$ and  $Y \deeq a b^{n+1} \cdot Z_3$.\par
Case 2\,: $(\,v_1, v_{\tau}, v_{t+2}\, ) = (\,a, c, b\, )$\\
We consider the case $\tau = t+1$, namely
\[
a X \deeq c W_{t+1} \deeq b Y.
\]
Applying the induction hypothesis to each step, we say that there exist positive words $Z_1$ and $Z_{2}$ such that
\[
X \deeq c Z_1,\,\, W_{t+1} \deeq a Z_1,\,\,\,\,\,\,\,\,\,\,\,\,\,\,
\]
\[
\,\,\,\,\,\,\,\,W_{t+1} \deeq b^{n} \cdot Z_{2},\,\, Y \deeq b^{n-1}a \cdot Z_{2}.
\]
Thus, we say that $a Z_1 \deeq  b^{n} \cdot Z_{2}$. Applying the induction hypothesis $(\mathrm{A})$ to this equation, we say that there exists a positive word $Z_3$ such that
\[
 Z_1 \deeq b^{n+1} \cdot Z_3,\,\, Z_2 \deeq ba \cdot Z_3.
\]
Hence, we have $X \deeq b \cdot b^{n}c \cdot Z_3$ and  $Y \deeq c \cdot b^{n}c \cdot Z_3$.\par
Case 3\,: $(\,v_1, v_{\tau}, v_{t+2}\, ) = (\,b, a, c\, )$\\
Then, we consider the following case 
\[
b X \deeq a W_{\tau} \deeq  c Y.
\]
Applying the induction hypothesis to each step, we say that there exist positive words $Z_1$ and $Z_2$ such that
\[
X \deeq c Z_1,\,\, W_{\tau} \deeq b Z_1,
\]
\[
W_{\tau} \deeq c Z_2,\,\, Y \deeq a Z_2.
\]
Moreover, we say that there exist a positive word $Z_3$ and an integer $k \in \Z_{\ge0}$ such that
\[
Z_1 \deeq c^{k} b^{n-1}a \cdot Z_3,\,\,Z_2 \deeq a^{k} b^{n} \cdot Z_3.
\]
Thus, we have
\[
X \deeq c^{k+1}b^{n-1}a \cdot Z_3,\,\,Y \deeq a^{k+1} b^{n} \cdot Z_3.
\]
\\
{\bf Step 2}.\,$\mathrm{H}_{r+1, h}$ for $0 \leq h \leq n-1$\\
We will show the general theorem $\mathrm{H}_{r+1, h}$ by induction on $h$. The case $h = 0$ is proved in Step 1. First, we show the case $h = 1$. Let $X^{(1)}$ be of word-length $r$, and let $Y$ be of word-length $r+1$. We consider a sequence of single transformations of $t+1$ steps
\[
V \cdot X^{(1)} \deeq \cdots \,\deeq c Y,
\]
where $V$ is a positive word of length $2$. We discuss three cases.  
\par
Case 1\,:\,$V = ba$.\\
We consider the following case
\begin{equation}
 ba \cdot X^{(1)} \deeq \cdots \,\deeq c Y.\,\,\,\,
\end{equation}
By the result of Step 1, we say that there exist a positive word $Z_1$ and an integer $k \in \Z_{\ge0}$ such that
\[
a X^{(1)} \deeq c^{k}b^{n-1}a \cdot Z_1,\,\, Y  \deeq a^{k}b^n \cdot Z_1. 
\]
Applying the induction hypothesis $(\mathrm{A})$, we say that there exists a positive word $Z_2$ such that
\[
 X^{(1)} \deeq c^{k} \cdot Z_2,\,\, b^{n-1}a \cdot Z_1  \deeq a Z_2. 
\]
Moreover, we say that there exists a positive word $Z_3$ such that
\[
b^{n-2}a \cdot Z_1 \deeq c Z_3,\,\,Z_2 \deeq b Z_3.
\]
By the induction hypothesis, we have a contradiction. Hence, there does not exist positive words $X^{(1)}$ and $Y$ that satisfy the equation $(4.1)$.
\par
Case 2\,:\,$V = bb$.\\
We consider the following case
\[
 bb \cdot X^{(1)} \deeq V_2 \cdot W_2 \deeq \cdots \deeq V_{t+1} \cdot W_{t+1} \,\deeq c Y,\,\,\,\,
\]
where $V_2$ and $V_{t+1}$ are positive words. It is enough to discuss the case $(\,V_2, V_{t+1}\,) = (\,bcb^{n}, ac\,)$. Applying the induction hypothesis $(\mathrm{A})$ to the equation 
\begin{equation}
bcb^{n} \cdot W_2 \deeq ac \cdot W_{t+1}, 
\end{equation}
we say that there exists a positive word $Z_1$ such that $c W_{t+1} \deeq  b Z_1$. Applying the induction hypothesis, we say that there exist a positive word $Z_2$ and an integer $k \in \Z_{\ge0}$ such that
\begin{equation}
W_{t+1} \deeq a^{k} b^{n} \cdot Z_2,\,\, Z_1 \deeq c^k b^{n-1}a \cdot Z_2. 
\end{equation}
Applying $(4.3)$ to the equation $(4.2)$, we have
\[
bcb^{n} \cdot W_2 \deeq ac \cdot a^{k} b^{n} \cdot Z_2. 
\]
Moreover, we say
\begin{equation}
b^{n} \cdot W_2 \deeq c^{k} b^{n-1}a \cdot Z_2. 
\end{equation}
We consider the following two cases.
\par
Case 2 -- 1\,:\,\,$k = 0$\\
There exists a positive word $Z_3$ such that
\[
W_2 \deeq c Z_3,\,\, Z_2 \deeq b Z_3. 
\]
Thus, we have
\[
\,\,\,X^{(1)} \deeq b^{n-1}a \cdot c Z_3 \deeq b^{n-2}a \cdot ba \cdot Z_3, 
\]
\[
Y \deeq ab^{n}b \cdot Z_3 \deeq b^{n}\cdot ba Z_3.\,\,\,\,\,\,\,\,\,\,\,\,\,\,\,\,\,\,\,\,\,
\]
\par
Case 2 -- 2\,:\,\,$k\ge1$\\
Applying the induction hypothesis to the equation $(4.4)$, we say that there exists a positive word $Z_3$ such that
\[
W_2  \deeq a^{k} \cdot Z_3. 
\]
Thus, we consider the equation $b^n \cdot Z_3 \deeq b^{n-1}a \cdot Z_2$.
We say that there exists a positive word $Z_4$ such that
\[
Z_2 \deeq b Z_4,\,\, Z_3 \deeq c Z_4.
\]
Thus, we have
\[
X^{(1)} \deeq b^{n-1}a \cdot a^{k} c \cdot Z_3 \deeq b^{n-2}a \cdot ba^{k+1} Z_3, 
\]
\[
Y \deeq a a^{k}b^{n}b \cdot Z_3 \deeq b^{n}\cdot ba^{k+1} Z_3.\,\,\,\,\,\,\,\,\,\,\,\,\,\,\,\,\,\,\,\,\,\,\,\,\,
\]
\par
Case 3\,:\,$V = bc$.\\
Then, we consider the following case 
\[
bc \cdot X^{(1)} \deeq \cdots \deeq  c Y.
\]
By the induction hypothesis, we say that there exist a positive word $Z_1$ and an integer $k \in \Z_{\ge0}$ such that
\[
c X^{(1)} \deeq c^k b^{n-1}a \cdot Z_1,\,\, Y \deeq a^k b^n \cdot Z_1.
\]
We consider the following two cases.
\par
Case 3 -- 1\,:\,\,$k = 0$\\
By the induction hypothesis, we say that there exists a positive word $Z_2$ such that
\[
X^{(1)} \deeq b^{n} c \cdot Z_2,\,\, Z_1 \deeq ba \cdot Z_2.
\]
Thus, we have
\[
X^{(1)} \deeq b^{n-1} a \cdot b Z_2,\,\,Y \deeq b^n ba \cdot Z_2 \deeq a b^n \cdot bZ_2.
\]
\par
Case 3 -- 2\,:\,\,$k \ge 1$\\
Then, we have 
\[
X^{(1)} \deeq c^{k-1} b^{n-1}a \cdot Z_1,\,\,Y \deeq a^k b^n \cdot Z_1.
\]
\par
Second, when $n \ge 4$, we show the theorem $\mathrm{H}_{r+1, h}$$\,(2 \leq h \leq n-2)$ by induction on $h$. We assume $h = 1, 2, \ldots , j \,(\,j \leq n-3\,)$. The case $h = 1$ has been proved. Let $X^{(j+1)}$ be of word-length $r-j$, and let $Y$ be of word-length $r+1$. We consider a sequence of single transformations of $t+1$ steps
\begin{equation}
V \cdot X^{(j+1)} \deeq \cdots \,\deeq c Y,
\end{equation}
where $V$ is a positive word of length $j+2$. We discuss the following three cases.  
\par
Case 1\,:\,$V \deeq bb^{j}a$.\\
Applying the induction hypothesis, we say that there exists a positive word $Z_1$ such that
\[
a X^{(j+1)} \deeq b^{n-j-1}a \cdot Z_1,\,\,Y \deeq b^{n} \cdot Z_1. 
\]
By the induction hypothesis, we say that there exists a positive word $Z_2$ such that
\[
X^{(j+1)} \deeq b Z_2,\,\, b^{n-j-2}a \cdot Z_1 \deeq c Z_2.
\]
By the induction hypothesis, we have a contradiction. Hence, there does not exist positive words $X^{(j+1)}$ and $Y$ that satisfy the equation $(4.5)$.
\par
Case 2\,:\,$V \deeq bb^{j+1}$.\\
Applying the induction hypothesis, we say that there exists a positive word $Z_1$ such that
\[
b X^{(j+1)} \deeq b^{n-j-1}a \cdot Z_1,\,\,Y \deeq b^{n} \cdot Z_1. 
\]
Thus, we have $X^{(j+1)} \deeq b^{n-j-2}a \cdot Z_1$.
\par
Case 3\,:\,$V \deeq bb^{j}c$.\\
Applying the induction hypothesis, we say that there exists a positive word $Z_1$ such that
\[
c X^{(j+1)} \deeq b^{n-j-1}a \cdot Z_1,\,\,Y \deeq b^{n} \cdot Z_1. 
\]
By the induction hypothesis, we have a contradiction. Hence, there does not exist positive words $X^{(j+1)}$ and $Y$ that satisfy the equation $(4.5)$.
\par
Lastly, we show the theorem $\mathrm{H}_{r+1, n-1}$. Let $X^{(n-1)}$ be of word-length $r-n+2$, and let $Y$ be of word-length $r+1$. We consider a sequence of single transformations of $t+1$ steps
\begin{equation}
V \cdot X^{(n-1)} \deeq \cdots \,\deeq c Y,
\end{equation}
where $V$ is a positive word of length $n$. We discuss the following three cases.  
\par
Case 1\,:\,$V \deeq b^{n-1}a$.\\
By the above result, we say that there exists a positive word $Z_1$ such that
\[
a X^{(n-1)} \deeq ba \cdot Z_1,\,\,Y \deeq b^{n} \cdot Z_1. 
\]
By the induction hypothesis, we say that there exists a positive word $Z_2$ such that
\[
X^{(n-1)} \deeq b a \cdot Z_2,\,\, Z_1 \deeq c Z_2.
\]
Thus, we have $Y \deeq b^{n}c \cdot Z_2$.
\par
Case 2\,:\,$V \deeq b^{n-1}b$.\\
By the above result, we say that there exists a positive word $Z_1$ such that
\[
b X^{(n-1)} \deeq ba \cdot Z_1,\,\,Y \deeq b^{n} \cdot Z_1. 
\]
Thus, we have $X^{(n-1)} \deeq a Z_1$.
\par
Case 3\,:\,$V \deeq b^{n-1}c$.\\
By the above result, we say that there exists a positive word $Z_1$ such that
\[
c X^{(n-1)} \deeq ba \cdot Z_1,\,\,Y \deeq b^{n} \cdot Z_1. 
\]
We have a contradiction. Hence, there does not exist positive words $X^{(n-1)}$ and $Y$ that satisfy the equation $(4.6)$. 

\end{proof}
This completes the proof of Theorem 4.1.
\end{proof}
Secondly, we show the cancellativity of $H^{+}_{n}$.
\ \ \\
\begin{theorem}
{\it The monoid $H_{n}^{+}$ is a cancellative monoid. }
 \par 
\end{theorem}
\begin{proof} First, we remark the following.\par
\begin{proposition}
{\it The left cancellativity on $H_{n}^{+}$ implies the right cancellativity.} 
 \par 
\end{proposition}
\begin{proof} Consider a map
 $\varphi:H^{+}_{n}\rightarrow H^{+}_{n}$,
 $W\mapsto \varphi(W):=\sigma$$(rev(W))$, where $\sigma$ is a
 permutation $\big(^{\,a, b, c}_{\,
c, b, a}\big)$. By following the proof in Proposition 4.2, we can show the statement.
\end{proof}
The following is sufficient to show the left cancellativity of the monoid $H^+_{n}$.
\begin{proposition}
{\it Let $Y$ be a positive word in $H^+_{n}$ of length $r\in \Z_{\ge0}$ and let $X^{(h)}$ be a positive word in $H^+_{n}$ of length $r-h \in \{\,2n, \ldots, r \,\}$.
\smallskip
\\
{\rm (i)}\, If $vX^{(0)} \deeq \! vY$ for some $v \in \{a, b, c \}$, then $X^{(0)} \deeq \! Y$.\\
{\rm (ii)}\, If $a X^{(0)} \deeq b Y$, then $X^{(0)} \deeq b Z$, $Y \deeq c Z$ for some positive word $Z$.\\
{\rm (iii)}\, If $a X^{(0)} \deeq c Y$, then $X^{(0)} \deeq c Z$, $Y \deeq a Z$ for some positive word $Z$.\\
{\rm (iv)}\, If $b X^{(0)} \deeq c Y$, then there exist an integer $k$ $(0 \leq k \leq r-2n-2)$ and a positive word $Z$ such that $X^{(0)} \deeq c^k (ab)^{n}ba \cdot Z$ and $Y \deeq  a^k b(ab)^n b \cdot Z$.\\
{\rm (v)}\, If $bb \cdot X^{(1)} \deeq c Y$, then $X^{(1)} \deeq c(ab)^{n-1}ba \cdot Z$, $Y \deeq b(ab)^n b \cdot Z$ for some positive word $Z$. \\
For $2 \leq h \leq r-2n$, we prepare the following propositions.\\
{\rm (vi\,-\,$h$)}\, If $c^{h-1}bb\cdot X^{(h)} \deeq b Y$, then $X^{(h)} \deeq c(ab)^{n-1}b \cdot Z$ and $Y \deeq (ab)^n b a^{h-1} \cdot Z$ for some positive word $Z$. \\
} \\
 \par 
\end{proposition}
\begin{proof} By refering to the double induction (see \cite{[G]}, \cite{[B-S]} for instance), we show the general theorem. The theorem for a positive word $Y$ of word-length $r$ and $X^{(h)}$ of word-length $r-h \in \{\,r-2n, \ldots, r \,\} $ will be refered to as $\mathrm{H}_{r, h}$. It is easy to show that, for $r = 0, 1$, $\mathrm{H}_{r, h}$ is true. For induction hypothesis, we assume\\
$(\mathrm{A})$\,\,$\mathrm{H}_{s, h}$ is true for $s = 0, \ldots, r$ and arbitrary $h$ for transformations of all chain-lengths, \\
and\\
$(\mathrm{B})$\,\,$\mathrm{H}_{r+1, h}$ is true for $0$ $\leq$ $h$ $\leq$ $\max\{ 0, r+1-2n \}$ for all chain-lengths $\leq$ $t$.\\
We will show the theorem $\mathrm{H}_{r+1, h}$ for chain-lengths $t+1$. For the sake of simplicity, we devide the proof into two steps.\\
{\bf Step 1}.\,$\mathrm{H}_{r+1, h}$ for $h = 0$\\
Let $X, Y$ be of word-length $r+1$, and let
\[
v_1 X \deeq v_2 W_2 \deeq \cdots \,\deeq v_{t+1} W_{t+1} \deeq v_{t+2} Y
\]
be a sequence of single transformations of $t+1$ steps, where $v_1, \ldots, v_{t+2} \in \{\,a, b, c\, \}$ and $W_2, \ldots , W_{t+1}$ are positive words of length $r+1$. By the assumption $t > 1$, there exists an index $\tau \in \{\,2, \ldots, t+1\,\}$ such that we can decompose the sequence into two steps
\[
v_1 X \deeq v_{\tau} W_{\tau} \deeq v_{t+2} Y,
\]
in which each step satisfies the induction hypothesis $(\mathrm{B})$.\par
If there exists $\tau$ such that $v_{\tau}$ is equal to either to $v_1$ or $v_{t+2}$, then by induction hypothesis, $W_{\tau}$ is equivalent either to $X$ or to $Y$. Hence, we obtain the statement for the $v_1 X \deeq v_{t+2} Y$. Thus, we assume from now on $v_{\tau} \not= v_1, v_{t+2}$ for $1 < \tau \leq t+1$.\par
Suppose $v_1 = v_{t+2}$. If there exists $\tau$ such that $(\,v_1 = v_{t+2}, v_{\tau}\, ) \not= (\,b, c\, ), (\,c, b\, )$, then each of the equivalences says the existence of $\alpha, \beta \in \{\,a, b, c\, \}$ and positive words $Z_1, Z_2$ such that $X \deeq \alpha Z_1$, $W_{\tau} \deeq \beta Z_1 \deeq \beta Z_2$ and $Y \deeq \alpha Z_2$. Applying the induction hypothesis $(\mathrm{A})$ to $\beta Z_1 \deeq \beta Z_2$, we get $Z_1 \deeq Z_2$. Hence, we obtain the statement $X \deeq \alpha Z_1 \deeq \alpha Z_2 \deeq Y$. Thus, we exclude these cases from our considerations. Next, we consider the case $(\,v_1 = v_{t+2}, v_{\tau}\, ) = (b, c )$. However, because of the above consideration, we say $v_2 = \cdots = v_{t+1} = c\,$. Hence, we consider the following case
\[
b X \deeq c W_1 \deeq \cdots \deeq c W_{t+1} \deeq b Y.
\]
Applying the induction hypothesis $(\mathrm{B})$ to each step, we say that there exist positive words $Z_3$ and $Z_4$ such that
\[
X \deeq (ab)^n ba \cdot Z_3,\,\, W_1 \deeq b(ab)^n b \cdot Z_3,\,\,\,\,
\]
\[
W_{t+1} \deeq b(ab)^n b \cdot Z_4,\,\, Y \deeq (ab)^n ba \cdot Z_4.
\]
Since an equation $W_1 \deeq W_{t+1}$ holds, we say that $X \deeq Y$.\\
In the case of $(\,v_1 = v_{t+2}, v_{\tau}\, ) = (c, b)$, we can prove the statement in a similar manner.\\
Suppose $v_1 \not= v_{t+2}$. It suffices to consider the following two cases.
\par
Case 1\,: $(\,v_1, v_{\tau}, v_{t+2}\, ) = (\,a, b, c\, )$\\
Because of the above consideration, we consider the case $\tau = t+1$, namely
\[
a X \deeq b W_{t+1} \deeq c Y.
\]
Applying the induction hypothesis to each step, we say that there exist positive words $Z_1$ and $Z_{2}$ such that
\[
X \deeq b Z_1,\,\, W_{t+1} \deeq c Z_1,\,\,\,\,\,\,\,\,\,\,\,\,\,\,\,\,\,\,\,\,\,\,\,\,\,\,\,\,\,\,\,\,
\]
\[
\,\,\,\,\,\,\,\,\,W_{t+1} \deeq (ab)^n ba \cdot Z_{2},\,\, Y \deeq b(ab)^n b \cdot Z_{2}.
\]
Thus, we say that $c Z_1 \deeq  (ab)^n ba \cdot Z_{2}$. Applying the induction hypothesis $(\mathrm{A})$ to this equation, we say that there exists a positive word $Z_3$ such that
\[
Z_1 \deeq a Z_3,\,\, b(ab)^{n-1} ba \cdot Z_{2} \deeq  c Z_3.
\]
Hence, we have $bbc(ab)^{n-2} ba \cdot Z_{2} \deeq  c Z_3$. Applying the induction hypothesis $(\mathrm{A})$ to this equation, there exists a positive word $Z_4$ such that
\[
c(ab)^{n-2} ba \cdot Z_2 \deeq c(ab)^{n-1} ba \cdot Z_4,\,\, Z_{3} \deeq b(ab)^n b\cdot Z_4.
\]
Hence, we have $ba \cdot Z_2 \deeq abba \cdot Z_4$. Moreover, we say that there exists a positive word $Z_5$ such that
\[
Z_2 \deeq cba \cdot Z_5,\,\,Z_4 \deeq c Z_5.
\]
Thus, we have
\[
X \deeq bab(ab)^n bc \cdot Z_5 \deeq c \cdot b(ab)^n bcb \cdot Z_5,
\]
\[
Y \deeq b(ab)^n bcba \cdot Z_5 \deeq a \cdot b(ab)^n bcb \cdot Z_5.
\]
\par
Case 2\,: $(\,v_1, v_{\tau}, v_{t+2}\, ) = (\,a, c, b\, )$\\
We consider the case $\tau = t+1$, namely
\[
a X \deeq c W_{t+1} \deeq b Y.
\]
Applying the induction hypothesis to each step, we say that there exist positive words $Z_1$ and $Z_{2}$ such that
\[
X \deeq c Z_1,\,\, W_{t+1} \deeq a Z_1,\,\,\,\,\,\,\,\,\,\,\,\,\,\,\,\,\,\,\,\,\,\,\,\,\,\,\,\,\,\,\,
\]
\[
\,\,\,\,\,\,\,\,W_{t+1} \deeq b(ab)^n b \cdot Z_{2},\,\, Y \deeq (ab)^{n}ba \cdot Z_{2}.
\]
Thus, we say that $a Z_1 \deeq b(ab)^n b \cdot Z_{2}$. Applying the induction hypothesis $(\mathrm{A})$ to this equation, we say that there exists a positive word $Z_3$ such that
\[
Z_1 \deeq b Z_3,\,\,(ab)^n b \cdot Z_2 \deeq c Z_3.
\]
Hence, there exists a positive word $Z_4$ such that
\[
b(ab)^{n-1} b \cdot Z_2 \deeq c Z_4,\,\,Z_3 \deeq a Z_4.
\]
We have $bbc(ab)^{n-2} b \cdot Z_2 \deeq c Z_4$. Applying the induction hypothesis $(\mathrm{A})$ to this equation, we say that there exists a positive word $Z_5$ such that
\[
c(ab)^{n-2}b \cdot Z_2 \deeq c(ab)^{n-1}ba \cdot Z_5,\,\,Z_4 \deeq b(ab)^n b \cdot Z_5.
\]
Hence, we have $Z_2 \deeq cba \cdot Z_5$. Thus, we have
\[
X \deeq cbab(ab)^n b \cdot Z_5 \deeq b(ab)^n bacb \cdot Z_5,
\]
\[
Y \deeq (ab)^{n}bacba \cdot Z_5 \deeq c(ab)^n bacb \cdot Z_5.
\]
\\
{\bf Step 2}.\,$\mathrm{H}_{r+1, h}$ for $1 \leq h \leq r+1-2n$\\
We will show the general theorem $\mathrm{H}_{r+1, h}$. First, we show the case $h = 1$. Then, we consider the following case
\[
bb\cdot X^{(1)} \deeq \cdots \deeq c Y.
\]
By the result of Step 1, we say that there exist a positive word $Z_1$ and an integer $k \in \Z_{\ge0}$ such that
\[
b X^{(1)} \deeq c^{k}(ab)^{n}ba \cdot Z_1,\,\, Y  \deeq a^{k}b(ab)^n b \cdot Z_1. 
\]
Thus, we have $b X^{(1)} \deeq ac^{k}b(ab)^{n-1}ba \cdot Z_1$. Applying the induction hypothesis $(\mathrm{A})$, we say that there exists a positive word $Z_2$ such that
\[
X^{(1)} \deeq c Z_2,\,\, b Z_2 \deeq c^{k}b(ab)^{n-1}ba \cdot Z_1 \deeq c^{k}bbc(ab)^{n-2}ba \cdot Z_1. 
\]
We consider the case $k \ge 1$. By the induction hypothesis, we say that there exists a positive word $Z_3$ such that
\[
Z_2 \deeq (ab)^n ba^k \cdot Z_3,\,\, c(ab)^{n-2}ba \cdot Z_1 \deeq c(ab)^{n-1}b \cdot Z_3. 
\]
Hence, we have $ba \cdot Z_1 \deeq abb \cdot Z_3$. Then, we have $a Z_1 \deeq cb \cdot Z_3$. By the induction hypothesis, there exists a positive word $Z_4$ such that
\[
Z_1 \deeq cb\cdot Z_4,\,\, Z_3 \deeq c Z_4.
\]
Thus, we have
\[
X^{(1)} \deeq c(ab)^n ba^k c\cdot Z_4 \deeq c(ab)^{n-1} ba \cdot cba^k \cdot Z_4,
\]
\[
Y \deeq a^{k}b(ab)^n bcb\cdot Z_4 \deeq b(ab)^n b\cdot cba^k\cdot Z_4.\,\,\,\,\,\,\,\,\,\,\,\,
\]
Next, we consider the case $2 \leq k \leq r+1-2n$. We consider the following case
\begin{equation}
c^{h-1}bb\cdot X^{(h)} \deeq \cdots \deeq b Y.
\end{equation}
By the result of Step 1, we say that there exist a positive word $Z_1$ and an integer $k_1 \in \Z_{\ge0}$ such that
\[
c^{h-2}bb\cdot X^{(h)} \deeq a^{k_1}b(ab)^n b\cdot Z_1,\,\, Y \deeq c^{k_1}(ab)^n ba \cdot Z_1.
\]
By repeating the same process $h-1$ times, there exist integers $k_{2}, \ldots, k_{h-1} \in \Z_{\ge0}$ and positive word $Z_{h-1}$  such that
\[
bb\cdot X^{(h)} \deeq a^{k_{h-1}}\cdot b(ab)^n b\cdot Z_{h-1}.
\]
Then, we have $b\cdot X^{(h)} \deeq c^{k_{h-1}}\cdot (ab)^n b\cdot Z_{h-1} \deeq ac^{k_{h-1}}\cdot b(ab)^{n-1} b\cdot Z_{h-1}$. By the induction hypothesis, there exists a positive word $Z_{h}$ such that
\[
X^{(h)} \deeq c Z_h,\,\, c^{k_{h-1}}\cdot b(ab)^{n-1} b\cdot Z_{h-1} \deeq b Z_h.
\]
Hence, we have $b Z_h \deeq c^{k_{h-1}}\cdot bbc(ab)^{n-2} b\cdot Z_{h-1}$. By the induction hypothesis, there exists a positive word $Z_{0}$ such that
\[
c(ab)^{n-2} b\cdot Z_{h-1} \deeq c(ab)^{n-1}b \cdot Z_0,\,\, Z_h \deeq (ab)^n ba^{k_{h-1}} \cdot Z_0.
\]
Thus, we have $b Z_{h-1} \deeq abb \cdot Z_0$. We have $Z_{h-1} \deeq cb \cdot Z_0$. Then, we have
\[
X^{(h)} \deeq c(ab)^n ba^{k_{h-1}} \cdot Z_0 \deeq c(ab)^{n-1}b \cdot cba^{k_{h-1}} \cdot Z_0.
\]
Applying this result to (4.7), we have
\[
bY \deeq c^{h-1}bb\cdot c(ab)^{n-1}b \cdot cba^{k_{h-1}} \cdot Z_0 \deeq b(ab)^n b a^{h-1} \cdot cba^{k_{h-1}} \cdot Z_0.
\]
Hence, we have $Y \deeq (ab)^n b a^{h-1} \cdot cba^{k_{h-1}} \cdot Z_0$.
\\
\end{proof}
This completes the proof of Theorem 4.4.
\end{proof}

\section{Calculations of the skew growth functions}
 In this section, we will calculate the skew growth functions for the monoids $G^{+}_{\mathrm{B_{ii}}}$, $G^{+}_{n}$, $H^{+}_{n}$ and $M_{\mathrm{abel}, m}$. \par 
First, we present an explicit calculation of the skew growth function for the monoid $G^{+}_{\mathrm{B_{ii}}}$. In \cite{[I1]}, we have made a success in calculating the growth function $P_{G^{+}_{\mathrm{B_{ii}}}, \mathrm{deg}}(t)$ by using the normal form for the monoid $G^{+}_{\mathrm{B_{ii}}}$. By the inversion formula, we can calculate the skew growth function $N_{G^{+}_{\mathrm{B_{ii}}}, \mathrm{deg}}(t)$. Nevertheless, we present an explicit calculation, because, in spite of the fact that the monoid is non-abelian and the height of it is infinite, we succeed in the non-trivial calculation.\\
{\bf Example.}\,\,{\bf 1.}\,\,First of all, we recall a fact from \cite{[S-I]} \S7.
\begin{lemma}
{\it  Let $X$ and $Y$ be positive words in $G^{+}_{\mathrm{B_{ii}}}$ of length $r\in \Z_{\ge0}$.
\smallskip
\\
{\rm (i)}\, If $vX \deeq \! vY$ for some $v \in \{a, b, c \}$, then $X \deeq \! Y$.\\
{\rm (ii)}\, If $a X \deeq b Y$, then $X \deeq b Z$, $Y \deeq c Z$ for some positive word $Z$.\\
{\rm (iii)}\, If $a X \deeq c Y$, then $X \deeq c Z$, $Y \deeq a Z$ for some positive word $Z$.\\
 {\rm (iv)}\, If $b X \deeq c Y$, then there exist an integer $k \in \Z_{\ge0}$ and a positive word $Z$ such that $X \deeq c^{k}ba \cdot Z$, $Y \deeq a^{k}bb \cdot Z$.}
 \par 
\end{lemma}
\par
Thanks to the Lemma 5.1, we have proved the cancellativity in \cite{[S-I]}. And we prove the following Lemma.

\begin{lemma}
{\it If an equation $bb\cdot X \deeq c Y$ in $G^{+}_{\mathrm{B_{ii}}}$ holds, then $X \deeq a Z$, $Y \deeq bb\cdot Z$ for some positive word $Z$. }
\end{lemma}
\begin{proof} Due to the Lemma 5.1, we say that there exist an integer $k \in \Z_{\ge0}$ and a positive word $Z_{0}$ such that
\begin{equation}
b X \deeq c^{k}ba \cdot Z_0,\,\,Y \deeq a^{k}bb \cdot Z_0. 
\end{equation}
We consider the case $k \ge 1$. Due to the Lemma 5.1, we say that there exist an integer $i_{1} \in \Z_{\ge0}$ and a positive word $Z_{1}$ such that
\[
X \deeq c^{i_{1}}ba \cdot Z_1,\,\,c^{k-1}ba \cdot Z_0 \deeq a^{i_{1}}bb \cdot Z_1. 
\]
Moreover, we say that there exists a positive word $Z^{(1)}_{0}$ such that
\[
Z_0 \deeq c^{i_{1}} \cdot Z^{(1)}_0,\,\,c^{k-1}ba \cdot Z^{(1)}_0 \deeq bb \cdot Z_1. 
\]
Repeating the same process $k$-times, there exist integers $i_{2}, \ldots, i_{k} \in \Z_{\ge0}$ and positive words $Z^{(k)}_0$ and $Z_{k}$ such that
\[
Z_0 \deeq c^{i_{1}+i_{2}+ \cdots +i_{k}} \cdot Z^{(k)}_0,\,\,ba \cdot Z^{(k)}_0 \deeq bb \cdot Z_k.
\]
Moreover, we say that there exists a positive word $Z'$ such that
\[
Z^{(k)}_0 \deeq b Z',\,\,Z_k \deeq c Z'. 
\]
Applying this result to $(5.1)$, we have
\[
bX \deeq c^{k}ba c^{i_{1}+i_{2}+ \cdots +i_{k}} b \cdot Z' \deeq ba c^{i_{1}+i_{2}+ \cdots +i_{k}} b a^{k} \cdot Z',
\]
\[
Y \deeq a^{k}bb c^{i_{1}+i_{2}+ \cdots +i_{k}} b\cdot Z' \deeq bb \cdot c^{i_{1}+i_{2}+ \cdots +i_{k}} ba^{k}\cdot Z'. 
\]
Thus, we have $X \deeq a \cdot c^{i_{1}+i_{2}+ \cdots +i_{k}} b a^{k} \cdot Z'$.
\end{proof}
As a consequence of Lemma 5.2, we obtain the followings.
\ \ \\
\begin{corollary}{\it If an equation $bb \cdot X \deeq c^{l} \cdot Y$ in $G^{+}_{\mathrm{B_{ii}}}$ holds for some positive integer $l$, then $X \deeq a^l \cdot Z$, $Y \deeq bb \cdot Z$ for some positive word $Z$.}
\end{corollary}
Due to the Corollary 5.3, we can solve the following equation.
\begin{proposition}
{\it If, for $0 \leq i < j$, an equation $c^{i}b \cdot X \deeq c^{j}b \cdot Y$ in $G^{+}_{\mathrm{B_{ii}}}$ holds, then there exist an integer $k \in \Z_{\ge0}$ and a positive word $Z$ such that
\[
X \deeq c^{k}ba^{j-i} \cdot Z,\,\,Y \deeq c^{k}b \cdot Z.
\]
} 
\end{proposition}
\begin{proof} Due to the cancellativity, we show $c^{i}b \cdot X \deeq c^{j}b \cdot Y \Leftrightarrow b X \deeq c^{j-i}b \cdot Y$. Thanks to the Lemma 5.1, we say that there exist an integer $k \in \Z_{\ge0}$ and a positive word $Z_1$ such that
\[
X \deeq c^{k}ba \cdot Z_1,\,\,c^{j-i-1}b \cdot Y \deeq a^{k}bb \cdot Z_1. 
\]
Moreover, we say that there exist $Y'$
\[
Y \deeq c^{k} \cdot Y',\,\,c^{j-i-1}b \cdot Y' \deeq bb \cdot Z_1.
\]
Due to the Corollary 5.3, there exists a positive word $Z_2$ such that
\[
bY' \deeq bb \cdot Z_2,\,\,Z_1 \deeq a^{j-1-1} \cdot Z_2.
\]
Thus, we have 
\[
X \deeq c^{k}ba^{j-i}\cdot Z_2,\,\, Y \deeq c^{k}b \cdot Z_2. 
\]
\end{proof}
As a corollary of the Proposition 5.4, we show the following lemma.
\begin{lemma}
{\it For $0 \leq \kappa_1 < \kappa_2 < \cdots < \kappa_{m} $,
\[
\mathrm{mcm}(\{\, c^{\kappa_1}b, c^{\kappa_2}b, \ldots, c^{\kappa_{m}}b\, \}) = \{\, c^{\kappa_{m}}b \cdot c^{k}b \mid k = 0, 1, \ldots \,\}
\]
 }
\end{lemma}
By using the Lemma 5.5, we easily show the following.
\begin{proposition}
{\it We have $h(G^+_{\mathrm{B_{ii}}}, \mathrm{deg}) = \infty$.
} 
\end{proposition}
\begin{proof} Due to the Proposition 5.1, we show
\[
\mathrm{mcm}(\{ b, c \}) = \{\, cb \cdot c^{k}b \mid k = 0, 1, \ldots \,\}. 
\]
Due to the Lemma 5.1, for $0 \leq \kappa_1 < \kappa_2 < \cdots < \kappa_{m}$, we say
\[
\mathrm{mcm}(\{\, cb \cdot c^{\kappa_1}b, cb \cdot c^{\kappa_2}b, \ldots, cb \cdot c^{\kappa_{m}}b\, \}) = \{\,cb \cdot c^{\kappa_{m}}b \cdot c^{k}b \mid k = 0, 1, \ldots \,\}.
\]
By using the Lemma 5.5 repeatedly, we show $h(G^+_{\mathrm{B_{ii}}}, \mathrm{deg}) = \infty$.
\end{proof}
By using the Lemma 5.5, we calculate the skew growth function. We have to consider four cases $J_1 = \{a, b\}, \{a, c\}, \{b, c\}, \{a, b, c\}$. The set $\mathrm{Tmcm}(G^{+}_{\mathrm{B_{ii}}}, J_1)$ denotes the set of all the towers starting from a fixed $J_1$. If $J_1 = \{a, b\}, \{a, c\}$, due to the Lemma 5.1, then $\mathrm{mcm}(\{a, b\})$ and $\mathrm{mcm}(\{a, c\})$ consist of only one element, respectively. Next, we consider the case $J_1 = \{b, c\}$. For a fixed tower $T$, if there exists an element $\Delta \in |T|$ such that $\mathrm{deg}(\Delta) =l+2$, then, from the Lemma 5.5, we say the uniqueness. For any fixed $l \in \Z_{>0}$, we calculate the coefficient of the term $t^{l+2}$ which is denoted by $a_{l}$, by counting all the signs $(-1)^{\#J_1+\cdots+\#J_{n}-n+1}$ in the definition $(3.1)$ associated with the towers $T = (I_0, J_1, J_2, \cdots, J_n)$ for which $\mathrm{deg}(\Delta)$ can take a value $l+2$. In order to calculate the $a_l$, we consider the set
\[ 
\mathcal{T}^{l}_{G^+_{\mathrm{B_{ii}}}} := \{\, T \in  \mathrm{Tmcm}(G^{+}_{\mathrm{B_{ii}}}, J_1) \mid \Delta \in |T| \,\, \mathrm{s.t.}\, \mathrm{deg}(\Delta)=l+2  \,\}. 
\] 
By using the Lemma 5.5 repeatedly, we show
\[ 
\max\{\mathrm{height}\,\, \mathrm{of}\,\, T \in \mathcal{T}^{l}_{G^+_{\mathrm{B_{ii}}}}\} = [(l+1)/2]. 
\]
For $u \in \{\, 1, \ldots, [(l+1)/2]\,\}$, we define the set
\[
\mathcal{T}^{l}_{G^+_{\mathrm{B_{ii}}}, u} := \{\, T \in  \mathrm{Tmcm}(G^{+}_{\mathrm{B_{ii}}}, J_1) \mid \mathrm{height}\,\, \mathrm{of}\,\, T = u,  \Delta \in |T|\,\, \mathrm{s.t.}\, \mathrm{deg}(\Delta)=l+2  \,\}.
\]
From here, we write $\mathcal{T}^{l}_{G^+_{\mathrm{B_{ii}}}}$ (resp. $\mathcal{T}^{l}_{G^+_{\mathrm{B_{ii}}}, u}$) simply by $\mathcal{T}^{l}$ (resp. $\mathcal{T}^{l}_{u}$). Thus, we have the decomposition:
\begin{equation}
\mathcal{T}^{l} = \bigsqcup_{\substack{u}} \mathcal{T}^{l}_{u}. 
\end{equation}
\par
{\bf Claim  1.} For any $u$, we show the following equality
\[
(-1)^{u-1} {}_{l-u} C _{u-1} = \sum_{T\in \mathcal{T}^{l}_{u}}(-1)^{\#J_1+\cdots+\#J_{u}-u+1}.
\]
\begin{proof}
For the case of $u = 1$, the equality holds. For the case of $u = 2$, we calculate the sum $\sum_{T\in \mathcal{T}^{l}_{2}}(-1)^{\#J_{2}-1}$. By indices $0 \leq \kappa_1 < \kappa_2 < \cdots < \kappa_{m}$, the set $J_2$ is generally written by $\{\, cb \cdot c^{\kappa_1}b, cb \cdot c^{\kappa_2}b, \ldots, cb \cdot c^{\kappa_{m}}b\, \}$. Due to the Lemma 5.5, we show that the maximum index $\kappa_{m}$ can range from $1$ to $l-2$. For a fixed index $\kappa_{m} = \kappa \in \{ 1, \ldots, l-2\}$, we easily show
\[
\sum_{T\in \mathcal{T}^{l}_{2}, \kappa_{m}=\kappa}(-1)^{\#J_{2}-1} = -1.
\]
Therefore, we show that the sum $\sum_{T\in \mathcal{T}^{l}_{2}}(-1)^{\#J_{2}-1} = -(l-2) = -{}_{l-2} C _{2-1}$. \par
We show the case for $3 \leq u \leq [(l+1)/2]$ by induction on $u$. We assume the case $u = j$. For the case of $u = j+1$, we focus our attention to the set $J_2$. If we write the set $J_2$ by $\{\, cb \cdot c^{\kappa_1}b, cb \cdot c^{\kappa_2}b, \ldots, cb \cdot c^{\kappa_{m}}b\, \}$, due to the the Lemma 5.5, we show that the maximum index $\kappa_{m}$ can range from $1$ to $l-2j$. By induction hypothesis, it suffices to show the following equality
\[
\sum^{l-2j}_{k=1} {}_{l-j-k-1} C _{j-1} = {}_{l-j-1} C _{j}.
\]
Therefore, we have shown the case $u = j+1$. This completes the proof.
\end{proof}

By the decomposition $(5.2)$, we show the following equality.
\par
{\bf Claim  2.} $a_{l}=\sum^{[(l-1)/2]}_{k=0}(-1)^k {}_{l-k-1} C _k$.\\ 
Then, we easily show the following.
\par
{\bf Claim  3.}  $a_{l+2} - a_{l+1} + a_{l} = 0$.
\begin{proof}
Since an equality ${}_{n+1} C _{k} - {}_{n} C _{k} = {}_{n} C _{k-1}$ holds, we can show our statement.
\end{proof}
We easily show $a_1 = a_2 = 1$. Hence, the sequence $\{a_{l}\}^{\infty}_{l=1}$ has a period $6$. Lastly, we consider the case $J_1 = \{a, b, c\}$. For any fixed $l \in \Z_{>0}$, we calculate the coefficient of the term $t^{l+3}$ which is denoted by $b_{l}$. Since $\mathrm{mcm}(\{a, b, c \}) = \{\, cb \cdot c^{k}b \mid k = 1, 2, \ldots \,\}$, we can reuse the Lemma 5.5. In a similar manner, we have the following conclusion.
\par
{\bf Claim  4.}  $b_{l+2} - b_{l+1} + b_{l} = 0$.\\
Since $b_1 = b_2 = 1$, we also show that the sequence $\{b_{l}\}^{\infty}_{l=1}$ has a period $6$. After all, we can calculate the skew growth function for the monoid $G^{+}_{\mathrm{B_{ii}}}$:
\[
N_{G^{+}_{\mathrm{B_{ii}}},\deg}(t) = 1-3t+2t^2+\frac{t^3}{1-t+t^2}-\frac{t^4}{1-t+t^2} = \frac{(1-t)^4}{1-t+t^2}.
\]
\par
Secondly, we present an explicit calculation of the skew growth function for the monoid $G^{+}_n$.\\
{\bf Example.}\,\,{\bf 2.}\,\,First of all, we show the following proposition.
\begin{proposition}
{\it If, for $0 \leq i < j$, an equation $c^{i}b^{n-1} \cdot X \deeq c^{j}b^{n-1} \cdot Y$ in $G^{+}_n$ holds, then there exists a positive word $Z$ such that
\[
X \deeq ba^{j-i} \cdot Z,\,\,Y \deeq b Z.
\]
} 
\end{proposition}
\begin{proof} Since we have shown the cancellativity in \S4, we show $c^{i}b^{n-1} \cdot X \deeq c^{j}b^{n-1} \cdot Y \Leftrightarrow b^{n-1}\cdot X \deeq c^{j-i}b^{n-1} \cdot Y$. Thanks to the Proposition 4.3\,{\rm (iv\,-\,$(n-2)$\,-\,$b$)}, we say that there exists a positive word $Z$ such that
\[
X \deeq ba^{j-i}\cdot Z,\,\, Y \deeq b Z.
\]
\end{proof}
As a corollary of the Proposition 5.7, we show the following lemma.
\begin{lemma}
{\it For $0 \leq \kappa_1 < \kappa_2 < \cdots < \kappa_{m} $,
\[
\mathrm{mcm}(\{\, c^{\kappa_1}b^{n-1}, c^{\kappa_2}b^{n-1}, \ldots, c^{\kappa_{m}}b^{n-1}\, \}) = \{ c^{\kappa_{m}}b^n \}
\]
 }
\end{lemma}
Thus, we obtain the following proposition.
\begin{proposition}
{\it We have $h(G^+_n, \mathrm{deg}) = 2$.
} 
\end{proposition}
By using the Lemma 5.8, we calculate the skew growth function. We have to consider four cases $J_1 = \{a, b\}, \{a, c\}, \{b, c\}, \{a, b, c\}$. The set $\mathrm{Tmcm}(G^{+}_n, J_1)$ denotes the set of all the towers starting from a fixed $J_1$. If $J_1 = \{a, b\}, \{a, c\}$, due to the Proposition 4.3, then $\mathrm{mcm}(\{a, b\})$ and $\mathrm{mcm}(\{a, c\})$ consist of only one element, respectively. Next, we consider the case $J_1 = \{b, c\}$. For any fixed $l \in \Z_{>0}$, we calculate the coefficient of the term $t^{n+l}$ which is denoted by $c_{l}$. In order to calculate the $c_l$, we consider the set
\[ 
\mathcal{T}^{l}_{G^{+}_n} := \{\, T \in  \mathrm{Tmcm}(G^{+}_n, J_1) \mid \Delta \in |T| \,\, \mathrm{s.t.}\, \mathrm{deg}(\Delta)=n+l  \,\}. 
\]
For $u \in \{ 1, 2\}$, we define the set
\[
\mathcal{T}^{l}_{G^{+}_n, u} := \{\, T \in  \mathrm{Tmcm}(G^{+}_n, J_1) \mid \mathrm{height}\,\, \mathrm{of}\,\, T = u,  \Delta \in |T|\,\, \mathrm{s.t.}\, \mathrm{deg}(\Delta)=n+l  \,\}.
\]
Since $\mathrm{mcm}(\{ b, c \}) = \{\, cb \cdot c^{k}b^{n-1} \mid k = 0, 1, \ldots \,\}$, we easily show $c_1 = c_2 = 1$. Moreover, we show the following.
\begin{proposition}
{\it We have $c_l = 0 \,\, (l = 3, 4, \ldots)$.
} 
\end{proposition}
\begin{proof} From the consideration in Claim 1 of Example 1, for $u = 2$, we also show
\[
\sum_{T\in \mathcal{T}^{l}_{G^{+}_n, u}}(-1)^{\#J_1+\cdots+\#J_{u}-u+1} = -1.
\]
Thus, we have $c_l = 0 \,\, (l = 3, 4, \ldots)$.
\end{proof}
Lastly, we consider the case $J_1 = \{a, b, c\}$. For any fixed $l \in \Z_{>0}$, we calculate the coefficient of the term $t^{n+l+1}$ which is denoted by $d_{l}$. In a similar way, we show $d_1 = d_2 = 1$ and $d_l = 0 \,\, (l = 3, 4, \ldots)$. After all, we calculate the skew growth function for the monoid $G^{+}_n$:
\[
N_{G^{+}_n,\deg}(t) = 1-3t+2t^2+(t^{n+1}+t^{n+2})-(t^{n+2}+t^{n+3})=(1-t)(t^{n+2}+t^{n+1}-2t+1).
\]
\begin{remark}
{\it By the inversion formula, we can calculate the growth function $P_{G^{+}_n,\deg}(t)$. As far as we know, it is difficult to calculate $P_{G^{+}_n,\deg}(t)$ directly.
} 
\end{remark}
\par
Thirdly, we present an explicit calculation of the skew growth function for the monoid $H^{+}_n$.\\
{\bf Example.}\,\,{\bf 3.}\,\,First of all, we show the following proposition.
\begin{proposition}
{\it If, for $0 \leq i < j$, an equation $c^{i}b(ab)^{n-1}ba \cdot X \deeq c^{j}b(ab)^{n-1}ba \cdot Y$ in $H^{+}_n$ holds, then there exists a positive word $Z$ such that
\[
X \deeq cba^{j-i} \cdot Z,\,\,Y \deeq cb \cdot Z.
\]
} 
\end{proposition}
\begin{proof} Since we have shown the cancellativity in \S4, we show $c^{i}b(ab)^{n-1}ba \cdot X \deeq c^{j}b(ab)^{n-1}ba \cdot Y \Leftrightarrow b(ab)^{n-1}ba\cdot X \deeq c^{j-i}b(ab)^{n-1}ba \cdot Y$. Thanks to the Proposition 4.6\,{\rm (vi\,-\,$h$)}, we say that there exists a positive word $Z_1$ such that
\[
(ab)^{n-1}ba\cdot X \deeq (ab)^n ba^{j-i}\cdot Z_1,\,\,c(ab)^{n-2}ba\cdot Y \deeq c(ab)^{n-1}b\cdot Z_1.
\]
Therefore, we say that there exists a positive word $Z_2$ such that
\[
X \deeq cba^{j-i}\cdot Z_2,\,\,Y \deeq cb\cdot Z_2.
\]
\end{proof}
As a corollary of the Proposition 5.11, we show the following lemma.
\begin{lemma}
{\it For $0 \leq \kappa_1 < \kappa_2 < \cdots < \kappa_{m} $,
\[
\mathrm{mcm}(\{\, c^{\kappa_1}b(ab)^{n-1}ba, c^{\kappa_2}b(ab)^{n-1}ba, \ldots, c^{\kappa_{m}}b(ab)^{n-1}ba\, \}) = \{ c^{\kappa_{m}}b(ab)^{n-1}bacb \}
\]
 }
\end{lemma}
Thus, we obtain the following proposition.
\begin{proposition}
{\it We have $h(H^+_n, \mathrm{deg}) = 2$.
} 
\end{proposition}
Thanks to the Lemma 5.12, we can calculate the skew growth function. We have to consider four cases $J_1 = \{a, b\}, \{a, c\}, \{b, c\}, \{a, b, c\}$. The set $\mathrm{Tmcm}(H^{+}_n, J_1)$ denotes the set of all the towers starting from a fixed $J_1$. If $J_1 = \{a, b\}, \{a, c\}$, due to the Proposition 4.6, then $\mathrm{mcm}(\{a, b\})$ and $\mathrm{mcm}(\{a, c\})$ consist of only one element, respectively. Next, we consider the case $J_1 = \{b, c\}$. For any fixed $l \in \Z_{>0}$, we calculate the coefficient of the term $t^{2n+3+l}$ which is denoted by $e_{l}$. In order to calculate the $e_l$, we consider the set
\[ 
\mathcal{T}^{l}_{H^{+}_n} := \{\, T \in  \mathrm{Tmcm}(H^{+}_n, J_1) \mid \Delta \in |T| \,\, \mathrm{s.t.}\, \mathrm{deg}(\Delta)=2n+3+l  \,\}. 
\]
For $u \in \{ 1, 2\}$, we define the set
\[
\mathcal{T}^{l}_{H^{+}_n, u} := \{\, T \in  \mathrm{Tmcm}(H^{+}_n, J_1) \mid \mathrm{height}\,\, \mathrm{of}\,\, T = u,  \Delta \in |T|\,\, \mathrm{s.t.}\, \mathrm{deg}(\Delta)=2n+3+l  \,\}.
\]
Since $\mathrm{mcm}(\{ b, c \}) = \{\, b c^{k}(ab)^{n}ba \mid k = 0, 1, \ldots \,\}$, we easily show $e_1 = e_2 = e_3 = 1$. Moreover, we show the following.
\begin{proposition}
{\it We have $e_l = 0 \,\, (l = 4, 5, \ldots)$.
} 
\end{proposition}
\begin{proof} From the consideration in Claim 1 of Example 1, for $u = 2$, we also show
\[
\sum_{T\in \mathcal{T}^{l}_{H^{+}_n, u}}(-1)^{\#J_1+\cdots+\#J_{u}-u+1} = -1.
\]
Thus, we have $e_l = 0 \,\, (l = 4, 5, \ldots)$.
\end{proof}
Lastly, we consider the case $J_1 = \{a, b, c\}$. For any fixed $l \in \Z_{>0}$, we calculate the coefficient of the term $t^{2n+4+l}$ which is denoted by $f_{l}$. In a similar way, we show $f_1 = f_2 = f_3 = 1$ and $f_l = 0 \,\, (l = 4, 5, \ldots)$. After all, we calculate the skew growth function for the monoid $H^{+}_n$:
\[
N_{H^{+}_n,\deg}(t) = 1-3t+2t^2+(t^{2n+3}+t^{2n+4}+t^{2n+5})-(t^{2n+4}+t^{2n+5}+t^{2n+6})
\]
\[
=(1-t)(t^{2n+5}+t^{2n+4}+t^{2n+3}-2t+1).\,\,\,\,\,\,\,\,\,\,\,\,\,\,\,\,\,\,\,\,\,\,\,\,\,\,\,\,\,\,\,\,\,\,
\]
\begin{remark}
{\it By the inversion formula, we can calculate the growth function $P_{H^{+}_n,\deg}(t)$. As far as we know, it is difficult to calculate $P_{H^{+}_n,\deg}(t)$ directly.
} 
\end{remark}
\par
Lastly, we calculate the skew growth function for the monoid $M_{\mathrm{abel}, m}$.\\
{\bf Example.}\,\,{\bf 4.}\,\,First of all, we easily show the following proposition.
\begin{proposition}
{\it Let $X$ and $Y$ be positive words in $M_{\mathrm{abel}, m}$ of length $r\in \Z_{\ge0}$.
\smallskip
\\
{\rm (i)}\, If $vX \deeq \! vY$ for some $v \in \{a, b \}$, then $X \deeq \! Y$.\\
{\rm (ii)}\, If $a X \deeq b Y$, then either $X \deeq a^{m-1}\cdot Z_1$ and $Y \deeq b^{m-1}\cdot Z_1$ for some positive word $Z_1$ or $X \deeq b Z_2$ and $Y \deeq a Z_2$ for some positive word $Z_2$.\\
} 
\end{proposition}
\begin{lemma}
{\it There exists a unique tower $T_n = (I_0, J_1, J_2, \cdots, J_n) $ of height $n \in \Z_{>0}$ with the ground set $I_0 = \{a, b \}$ such that
\[
J_{2k-1} = \{a^{(k-1)m+1}, a^{(k-1)m}b \}\,\,(k = 1, \ldots,[(n+1)/2]),
\]
\[
J_{2k} = \{a^{km}, a^{(k-1)m+1}b \}\,\,(k = 1, \ldots, [n/2]).\,\,\,\,\,\,\,\,\,\,\,\,\,\,\,\,\,\,\,\,\,\,\,\,\,\,\,\,\,\,\,\,
\]
 }
\end{lemma}
\begin{proof} We easily show $J_1 = \{a, b \}$ and $J_2 = \{a^m, ab \}$. Thanks to the Proposition 5.15, we show our statement by induction on $k$.
\end{proof}
Therefore, we immediately show $h(M_{\mathrm{abel}, m}, \mathrm{deg}) = \infty$. And, from the definition $(3.1)$, we can calculate the skew growth function 
\[
N_{M_{\mathrm{abel}, m}, \deg}(t) = (1-2t+t^2)(1+t^m+t^{2m}+\cdots) = \frac{(1-t)^2}{1-t^m}.
\]
\section{Appendix}
In this section, we present three examples that suggest the ralationship between the form of the spherical growth function for a monoid ${\langle L \mid R\,\rangle}_{mo}$ and properties of the corresponding group ${\langle L \mid R\,\rangle}$. For the three examples, by observing the distribution of the zeroes of the denominator polynomials of the growth functions for them, we conjecture that the corresponding groups contain free abelian subgroups of finite index.\\
{\bf Example.}\,\,{\bf 1.}\,\,We recall an example, the monoid $G^+_{\mathrm{B_{ii}}}$, from \cite{[I1]}. By using the normal form of the monoid $G^+_{\mathrm{B_{ii}}}$, the author has calculated the spherical growth function $P_{G^+_{\mathrm{B_{ii}}}, \mathrm{deg}}(t)$ for the monoid. The spherical growth function $P_{G^+_{\mathrm{B_{ii}}}, \mathrm{deg}}(t)$ can be expressed as a rational function $\frac{1-t+t^2}{(1-t)^4}$. Since the zeroes of the denominator polynomial of $P_{G^+_{\mathrm{B_{ii}}}, \mathrm{deg}}(t)$ only consists of $1$ with multiplicity $4$, it is conjectured that the corresponding group $G_{\mathrm{B_{ii}}}$ contains a free abelian subgroup of rank $4$ of finite index. Indeed, the author has shown the following.
\begin{proposition}
{\it The followig {\rm (i)},{\rm (ii)} and {\rm (iii)} hold.
\smallskip
\\
{\rm (i)}\, The group $G_{\mathrm{B_{ii}}}$ contains a subgroup of index three isomorphic to $\Z^4$.\\
{\rm (ii)}\, The group $G_{\mathrm{B_{ii}}}$ has a polynomial growth rate.\\
{\rm (iii)}\, The group $G_{\mathrm{B_{ii}}}$ is a solvable group.
} 
\end{proposition}

\noindent
{\bf Example.}\,\,{\bf 2.}\,\,We consider the following example
\[
\begin{array}{lll}
\biggl{\langle}
a,b,c\,
\biggl{|}
\begin{array}{cc}cb=ba,\\
 ab=bc,\\
 ac=ca
\end{array}
\biggl{\rangle}_{mo} .
\end{array}
\]
We easily show the cancellativity of it by refering to the double induction (see \cite{[G]}). The spherical growth function can be expressed as a rational function $\frac{1}{(1-t)^3}$. It is also conjectured that the corresponding group contains a free abelian subgroup of rank $3$ of finite index. Indeed, we can show that corresponding group contains a subgroup of index two isomorphic to $\Z^3$. Moreover, we show that the group has a polynomial growth rate and is a solvable group.
\\
{\bf Example.}\,\,{\bf 3.}\,\,We consider the following example
\[
\begin{array}{lll}
\biggl{\langle}
a,b,c,d\,
\biggl{|}
\begin{array}{cc}ab=bc, ac=ca,\\
 cb=ba, bd=db,\\
 ad=dc, cd=da
\end{array}
\biggl{\rangle}_{mo} .
\end{array}
\]
We easily show the cancellativity of it by refering to the double induction (see \cite{[G]}). The spherical growth function can be expressed as a rational function $\frac{1}{(1-t)^4}$. It is also conjectured that the corresponding group contains a free abelian subgroup of rank $4$ of finite index. Indeed, we can show that corresponding group contains a subgroup of index three isomorphic to $\Z^4$. Moreover, we show that the group has a polynomial growth rate and is a solvable group.
\\
\emph{Acknowledgement.}\! 
The author is grateful to Kyoji Saito for very interesting discussions and encouragement. This research is supported by JSPS Fellowships for Young Scientists $(23\cdot10023)$. This researsh is also supported by World Premier International Research Center Initiative (WPI Initiative), MEXT, Japan.

\begin{flushright}
\begin{small}
Kavli IPMU, \\
University of Tokyo, \\
Kashiwa, Chiba 277-8583 Japan \\

e-mail address :  tishibe@ms.u-tokyo.ac.jp
\end{small}
\end{flushright}

\begin{thebibliography}{}





\bibitem[A-N]{[A-N]} M. Albenque and N. Philippe:  
 Growth function for a class of monoids, 
FPSAC 2009, DMTCS proc. 25-38.

\bibitem[B]{[B]} M. Brazila: 
Growth functions for some one-relation monoids, communications in Algebra vol.\textbf{21}, (1993) 3135-3146.
\bibitem[Bro]{[Bro]} A. Bronfman:
 Growth functions of a class of monoids, preprint, 2001.
\bibitem[B-S]{[B-S]} E. Brieskorn and K. Saito:   
 Artin-Gruppen und Coxeter-Gruppen, 
{\it Inventiones Math.} \textbf{17} (1972) 245-271, English translation by C.~Coleman, R.~Corran, J.~Crisp, D.~Easdown, R.~Howlett, D.~Jackson and A.~Ram at the University of Sydney, 1996.






\bibitem[Deh1]{[Deh1]} 
P. Dehornoy:
Complete positive group presentations, J. Algebra {\bf 268} (2003) 156-197.

\bibitem[Deh2]{[Deh2]} 
P. Dehornoy:
The subword reversing method, Intern. J. Alg. and Comput. {\bf 21} (2011) 71-118.
\bibitem[De]{[De]} P. Deligne: 
 Les immeubles des tresses g\'{e}n\'{e}raliz\'{e}, 
{\it Inventiones Math.} \textbf{17} (1972) 273-302. 




\bibitem[G]{[G]}F.A. Garside:
 The braid groups and other groups, Quart.~J.~Math.~Oxford,~{\bf 20} (1969), 235-254.



\bibitem[I1]{[I1]} T. Ishibe:
 On the monoid in the fundamental group of type $\mathrm{B_{ii}}$, Hiroshima Mathematical Journal Vol.42 No.1(March) 2012.

\bibitem[I2]{[I2]} T. Ishibe:
 Infinite examples of cancellative monoids that do not always have least common multiple, submitted.

\bibitem[I3]{[I3]} T. Ishibe:
 Infinite examples of cancellative monoids that do not always have least common multiple $\mathrm{II}$, in preparation.

\bibitem[K-T-Y]{[K-T-Y]} K. Kobayashi, S. Tsuchioka and S. Yasuda:
 Partial theta and growth series of Artin monoids of finite type, in preparation.

\bibitem[S1]{[S1]}
 K. Saito:
 Inversion formula for the growth function of a cancellative monoid, submitted.

\bibitem[S2]{[S2]} 
 K. Saito:  
Growth functions associated with Artin monoids of finite type, 
Proc. Japan Acad. Ser. A Math. Sci. \textbf{84} (2008), no.10, 179-183.

\bibitem[S3]{[S3]} 
 K. Saito:  
Growth functions for Artin monoids, 
Proc. Japan Acad. Ser. A Math. Sci. \textbf{85} (2009), no.7, 84-88.
\bibitem[S4]{[S4]} 
 K. Saito:  
Limit elements in the Configuration Algebra for a Cancellative Monoid, 
Publ. RIMS Kyoto Univ. \textbf{46} (2010), 37-113.

\bibitem[S5]{[S5]} 
K. Saito:  
Growth partition functions for cancellative infinite monoids, 
preprint RIMS-1705 (2010).
\bibitem[S-I]{[S-I]}
K. Saito and T. Ishibe:
 Monoids in the fundamental groups of the complement of logarithmic free divisors in $\C^3$, Journal of Algebra 344 (2011), 137-160.

\bibitem[Xu]{[Xu]}
P. Xu:
 Growth of the positive braid semigroups, Journal of Pure and Applied Algebra \textbf{80} (1992), no.2, 197-215.






\end{thebibliography}
\end{document}